\documentclass[12pt]{amsart}
\usepackage{euscript,amsmath,amsthm,amsfonts,amssymb,amsrefs}
\usepackage{csquotes}

\MakeOuterQuote{"}

\newcommand{\R}{\mathbb{R}}

\numberwithin{equation}{section}

\newtheorem{theorem}{Theorem}[section]

\newtheorem{lemma}[theorem]{Lemma}

\theoremstyle{definition}
\newtheorem*{nota}{Notation}

\begin{document}

\title{Embedding Heegaard Decompositions}

\author{Ian Agol}
\address{\hskip-\parindent
	Ian Agol\\
	Department of Mathematics\\
	University of California, Berkeley\\
	Berkeley, CA 94720\\}
\email{ianagol@math.berkeley.edu}

\author{Michael Freedman}
\address{\hskip-\parindent
	Michael Freedman\\
    Microsoft Research, Station Q, and Department of Mathematics\\
    University of California, Santa Barbara\\
    Santa Barbara, CA 93106\\}
\email{mfreedman@math.ucsb.edu}

\begin{abstract}
A smooth embedding of a closed $3$-manifold $M$ in $\mathbb{R}^4$ may generically be composed with projection to the fourth coordinate to determine a Morse function on $M$ and hence a Heegaard splitting $M=X\cup_\Sigma Y$.  However, starting with a Heegaard splitting, we find an obstruction coming from the geometry of the curve complex $C(\Sigma)$ to realizing a corresponding embedding $M\hookrightarrow \mathbb{R}^4$.
\end{abstract}

\maketitle

\section{Introduction}

The tools for showing that a closed 3-manifold $M$ does not smoothly embed in $\R^4$ seem rather primitive. There does not seem to be any known 3-manifold $M$ which embeds in some integral homology 4-sphere $\Sigma^4$ and is known not to embed in $\R^4$. But tools for this would be highly desirable since Budney and Burton's  3-manifold survey turns up four examples of closed 3-manifolds $M$ embedding in a homotopy 4-sphere for which no embedding in $\R^4$ is known \cite[\textsection 4]{Budney}. This raises the possibility that 3-manifold embeddings could be used to detect exotic structures (if there are homotopy 4-balls which do not embed in $\mathbb{R}^4$).

Our goal here is to find a bridge between the rich subject of surface dynamics, e.g. the mapping class group, and embeddability in the hope that the coordinate structure of $\R^4$ will make an essential appearance. We are partially successful. We find a robust connection between the very coarse "handlebody metric" $d_H$ on the curve complex $C(\Sigma)$ recently studied in \cite{MS} and embeddability of the corresponding Heegaard decompositions. Said another way, we turn $d_H$ into an obstruction to embedding $f: M \hookrightarrow \R^4$ with the fourth coordinate already prescribed. We have known this result since 2012 but have been unable to accomplish the obvious next step: find some residual obstruction which is independent of the fixed 4\textsuperscript{th} coordinate function, i.e. define a true embedding obstruction based on surface dynamics. Since \cite{MS} has now appeared in print and our argument provides a simple application, possibly with yet unrealized potential, we present it here.

Assume that we have been given a Morse function $f:M\rightarrow \mathbb{R}$, how can we show it is not the fourth coordinate of any embedding in $\mathbb{R}^4$? Actually, the obstruction we formulate will not make use of the entire data of a Morse function but merely the Heegaard decomposition $M=X\cup_\Sigma Y$ canonically determined (up to isotopy) by $f$.  The handlebody $X$ is a neighborhood of the ascending manifolds of critical points of index $=2$ and $3$ and $Y$ is a neighborhood of the descending manifold of critical points of index $0$ and $1$.

Our chief tool is the curve complex $C(\Sigma)$ \cite{Harvey} and its metrics.  The vertices of $C(\Sigma)$ are isotopy classes of simple closed curves (sccs) on $\Sigma$, and Hempel \cite{Hempel} introduced the metric $d$, the largest metric where disjoint sccs have distance $=1$.  We will exploit a much coarser metric $d_H$, ``handlebody distance,'' defined as the largest metric where any two sccs bounding disks in the same handlebody $H$, $\partial H=\Sigma$, have distance $=1$.  This distance is easily seen to be quasi-isometric to the ``electrification metric'' $d_E$ recently introduced in \cite{MS}, where it is proved that $\text{diam}_{d_E}(C(\Sigma))=\infty$, for genus $\Sigma\ge 2$.  So for us, a key fact will be

\begin{equation}
\text{for genus }\Sigma\ge 2, \text{ diam}_{d_H}(C(\Sigma))=\infty.
\end{equation}

Let $D(X)\subset C(\Sigma)$ be the set of sccs bounding disks in the handlebody $X=\partial \Sigma$.

We prove the following:

\vspace{.1in}
\noindent{\bf{Theorem 3.1.}}  \emph{If} $M$ \emph{embeds in} smoothly $\mathbb{R}^4$ \emph{with one (say, the fourth) coordinate determining the Heegaard decomposition} $M=X\cup_\Sigma Y$\emph{, then }$d_H(D(X),D(Y))\le 2$.
\vspace{.1in}

We actually supply two proofs (using slightly different techniques, yielding slightly different constants, and supporting different generalizations).

\section{Ambient Morse Theory}

This section recalls an ``ambient'' version of Morse theory appropriate to embedded submanifolds $M^3\hookrightarrow \mathbb{R}^4$.

When speaking of an embedding $f: M \hookrightarrow \R^4$, we will feel free to change the target space to $S^4$ or $S^3 \times \R$, by adding or deleting a point, without renaming the map or calling other attention to the change.

Suppose we are given a codimension-1 smooth embedding $g:M^3\hookrightarrow \mathbb{R}^4$ of a closed connected $3$-manifold with fourth coordinate $g_4=f$.  Using an elementary general position argument, one constructs an isotopy from $g$ to $g'$ with the critical points of the fourth coordinate $g'_4$ occurring in \emph{order} (higher index critical points take larger values).  Such Morse functions will be called \emph{ordered}, see Lemma \ref{ordered} for a sketch of proof.

Some choices are made in this procedure which could influence the order of handle attachments but not the diffeomorphism type of the Heegaard decomposition $M=X\cup_\Sigma Y$, where $Y=\cup$ (handles of index $=0,1$) and $X=\cup$ (handles of index $=2,3$). The topology of $X$ ($Y$) relative to $\Sigma$ is, however, independent of any choices.

\begin{lemma}\label{schm:2.2}
If $M$ is an embedding in $\mathbb{R}^4$ and $M=X\cup_\Sigma Y$ is the Heegaard splitting associated to the fourth coordinate with $\Sigma\subset \mathbb{R}^3\times 0\subset \mathbb{R}^4$, then provided $\mathrm{genus}(\Sigma) \geq 1$, $D(X)$ and $D(Y)$ both contain at least one scc which compresses in $\mathbb{R}^3\times 0$.\qed
\end{lemma}

\begin{lemma}\label{schm:2.3}
If $N$ is a $3$-manifold (compact or non-compact) and there is an embedding $M\hookrightarrow N\times \mathbb{R}$ inducing a Heegaard splitting $M=X\cup_\Sigma Y$, $\Sigma\subset N\times 0$, then provided $\mathrm{genus}(\Sigma) \geq 1$, $D(X)$ and $D(Y)$ both contain at least one scc which compresses in $N\times 0$.\qed
\end{lemma}

\begin{proof}
	Both lemmas are proven by sliding $\Sigma$ up and down the gradient lines of the Morse function until the first collapse of an essential scc in $\Sigma$ is observed.
\end{proof}

\section{Two Distance Estimates}

\begin{proof}[Proof of Theorem 3.1]
We may compactify horizontal slices to consider the embedding of $M$ as into $S^3\times \mathbb{R}$.  $M=X\cup_\Sigma Y$ and we may assume
$$\Sigma\subset S^3=S^3\times 0\subset S^3\times\mathbb{R}$$
and, since lens spaces do not embed in $\R^4$, without loss of generality that $g(\Sigma)\ge 2$.  Note that $\Sigma\subset S^3$ may not be a Heegaard surface for $S^3$ but by Lemma \ref{schm:2.2} must contain at least one scc of $D(X)$ and one scc of $D(Y)$ (which might be identical) which compress into $S^3$.

\begin{nota}
$S^3=A\cup_\Sigma B$ and $V\subset A$ and $W\subset B$ are maximal compression bodies for $\Sigma$ in $A$ and $B$, respectively (see \cite{CG96} for the definition of a compression body).
\end{nota}

At most one of $V$ and $W$ is a product collar (since $S^3$ is non-Haken).  By Fox \cite[Main Theorem (1)]{Fox98}, $V\cup_\Sigma W$ may be re-embedded into $S^3$ so that the complementary regions $P$ and $Q$ lying in $A$ and $B$ respectively are unions of handlebodies and consequently $H=P\cup_{\partial_-V}V$ and $J=Q\cup_{\partial_-W}W$ are also handlebodies and $S^3=H\cup_\Sigma J$ is a Heegaard decomposition.  By Waldhausen \cite{Waldhausen}, this decomposition is standard; thus, $D(H)\cap D(J)\neq \emptyset$, so

\begin{equation}
    \text{diam}_{d_H}(D(H)\cup D(J))=2.
\end{equation}
(From the definition of $d_H$, $\text{diam}_{d_H}(D(H)\cup D(J))\le 2$, but since $D(H)\cup D(J)$ contains a pair of sccs with homological intersection number $=1$, $\text{diam}_{d_H}(D(H)\cup D(J))>1$.)

But as noted above, both $D(X)$ and $D(Y)$ must meet $D(H)\cup D(J)$, so
$$d_H(D(X),D(Y))\le 2.$$
\end{proof}

\setcounter{theorem}{1}
\begin{theorem}\label{thm:3.2}
Let $N$ be a closed, reducible, $3$-manifold containing no incompressible surface, or $S^3$.  Let $M$ be a closed $3$-manifold embedded in $N\times \mathbb{R}$, with the $\mathbb{R}$-coordinate inducing a Heegaard splitting $M=X\cup_\Sigma Y$.  We have $d_H(D(X),D(Y))\le 3$.
\end{theorem}

\begin{proof}
Use ambient Morse theory to produce $\Sigma \subset N\cong N\times 0$ with $X\subset N\times[0,\infty)$ and $Y\subset N\times (-\infty,0]$.  By Lemma \ref{schm:2.3}, there is at least one scc of $D(X)$ and $D(Y)$ compressing in $N$.  (It might be the same curve.)  Let $\Sigma$ separate $N$ as $N=A\cup_\Sigma B$ and let $V\subset A$ and $W\subset B$ be the maximal compression bodies on the two sides of $\Sigma$.  Since $N$ contains no incompressible surface, at most one of $V$ and $W$ is a product collar.

Suppose neither $V$ nor $W$ is a collar.  Then it must follow that the $d$-distance between the sets of compression disks, $d(D(V),D(W))\le 1$.  For otherwise the generalized Heegaard decomposition $V\cup_\Sigma W$ would be strongly irreducible, so the main theorem of C--G \cite{CG96} implies that $V\cup_\Sigma W$ is irreducible and has either incompressible or empty boundary.  This boundary must be non-empty by the reducibility of $N$.  But since $\partial_-V\hookrightarrow (A\backslash V)$ and $\partial_-W\hookrightarrow (B\backslash W)$ are incompressible (by maximality of $V$ and $W$), we would conclude $\partial(V\cup_\Sigma W)\hookrightarrow N$ is incompressible, contradicting the assumption of no incompressible surface.  However, by Lemma \ref{schm:2.3}, we have
$$D(X)\cap (D(V)\cup D(W))\neq\emptyset\neq D(Y)\cap(D(V)\cup D(W))$$

Since $d_H(D(V),D(W))\le d(D(V),D(W))\le 1$ and $\mathrm{diam}_{d_H}(D(V))$, $\mathrm{diam}_{d_H}(D(W))$ $\le 1$, we conclude by the triangle inequality that $d_H(D(X),D(Y))\le 3$.

Now suppose one of $V$ or $W$, say $V$, is a collar.  Then
$$D(X)\cap D(W)\neq\emptyset\neq D(Y)\cap D(W),$$
so $d_H(D(X),D(Y))\le 1$.
\end{proof}

Maher and Schleimer studied \cite{MS} a metric $d_E$ which is clearly quasi-isometric to $d_H$.  $d_E$ is defined by adding a new vertex $h$ to $C(\Sigma)$ for each handlebody $H$, $\partial H=\Sigma$, and adjoining length $=1$ edges between each scc in $D(H)$ and $h$.  They prove that for genus $\Sigma > 2$, $\mathrm{diam}_{d_E} = \infty$. Thus we have:

\begin{theorem}\label{thm:3.4}
$\mathrm{diam}_{d_H}(C(\Sigma))=\infty$, for $\mathrm{genus}(\Sigma)\geq 2$.\qed
\end{theorem}

Since $\text{diam}_{d_H}(C(\Sigma))=\infty$, Theorems 3.1 and \ref{thm:3.2} obstruct certain---in some sense, 
most---Morse functions, or Heegaard decompositions $M=X\cup_\Sigma Y$ from arising via embedding 
$M\hookrightarrow N\times\mathbb{R}$ for $N = S^3$, or more generally $N$ closed, reducible, 
and containing no incompressible surface.

\section{Appendix}

We did not find a good reference in the literature so provide here a sketch of the well-known folk theorem on ambient Morse theory.

\begin{lemma} \label{ordered}
Let $e:M \to  N\times\mathbb{R}$ be a smooth embedding where $M$ and $N$ are closed $d$-dimensional smooth manifolds and $\mathbb{R}$ the real numbers. $e$ is smoothly isotopic to an embedding $eÕ$ so that projection onto $\mathbb{R}$ is a Morse function enjoying a weak self-indexing property: if critical points $x$ and $y$ have index $p$ and $q$ respectively, and if $p < q$, then $eÕ(x)  < eÕ(y)$.
\end{lemma}
\begin{proof}
The composition $M \to N\times\mathbb{R} \to \mathbb{R}$ may be perturbed to be a Morse function \cite[\textsection 6]{Milnor} and choosing such a perturbation small enough we may use it to redefine the last coordinate of $e$ while keeping the embedded property; call the result $e''$. $e''$ obeys the first conclusion of the lemma; we need now to order the critical points.
To do this mark the ascending manifold (with respect to gradient of the composition) of each critical point $p$. If $p$ has index $i$ then this ascending manifold is an (improperly) embedded copy of $\mathbb{R}^{d-i}$.  We now construct the required isotopy starting with critical points of index zero. Leave the lowest valued index zero points alone call their value (or ÒheightÓ) $0$. We wish to push any higher valued index zero points down near $0$, proceeding in order of height. Pushing such a point straight down, we may encounter another sheet of $e''$. By general position we can perturb the downward push so that we do not encounter any ascending manifold (AM) except possibly the AM of an index zero critical point. Now bend the push to skim above the sheet until the index zero critical point (CP) generating the AM is nearly reached and stop the push there.  At this point all we have done is isotoped $e''$ (but do not change the notation) to make all the index zero critical points have lower values that the higher index CPÕs.   
Now turn to any lowest valued index $1$ CPÕs and push it down below higher index critical points. Now, again in height order, treat the other index $1$ critical points. ``treating'' means that we take the descending manifold (DM) ( an arc in this case) and push it, relative its ends, straight down until it nears another sheet of $e''$. We are concerned about the DM meeting an AM of another critical point of index $> 0$. (We are not concerned about the AMs of index $0$ CPÕs because we will not push the DM that low). Again by general position we only need encounter AM of critical points of index $1$ and when this happens the push, again bent to skim the sheet following the gradient, will need to stop just before a collision occurs with a lower index $1$ critical point. This push can and should be done so that the height function along the DM (arc) retains the property of having a single CP. This means that as we proceed the number of critical points is never changed after the initial perturbation to $e''$. Now we have all the critical points on index $0$ and $1$ arrange with the desired monotonicity.
The preceding paragraph is applicable verbatim to critical points of index $2$, then $3$, É, then $d-1$. 
It is not required to touch the neighborhoods of the local maxima once the lower index CPÕs are appropriately ordered.

Notice that the proof does {\it not} allow us to reorder to our pleasure critical values within a fixed index: there is a very real phenomenon of ÒnestingÓ of index $k$ critical points and as a general rule that order cannot be rearranged by isotopy. Imagine stacked cups or saddles or chips or hats, which cannot be wiggled to be rearranged in height (otherwise the stack would tumble down). If is worth mentioning here that the same lemma (and essentially the same proof) still holds if $dim N > dim M$.  However in this case ( codimension-$2$ and higher embeddings) the ordering of critical values may be arbitrarily rearranged in the isotopy: the nesting disappears.
\end{proof}

\section*{Acknowledgements}

\noindent
The first author acknowledges support by NSF grant DMS-1105738.

\vspace{1em}
\noindent
The second author would like to acknowledge stimulating discussion with Cameron Gordon on this subject.

\end{document}